\documentclass[reqno]{amsart}

\usepackage[sorting=none, sorting=nyt, maxbibnames=99, backend=biber]{biblatex}

\renewbibmacro{in:}{}
\bibliography{references.bib}
\usepackage{amssymb}
\usepackage{mathrsfs}
\usepackage{graphics,graphicx, mathrsfs}
\usepackage{enumerate}
\usepackage{enumitem}
\usepackage{textcomp}
\usepackage{url}
\usepackage{xcolor}
\definecolor{vio}{rgb}{0.54, 0.17, 0.89}
\usepackage[ruled, lined, linesnumbered, longend]{algorithm2e}
\usepackage{hyperref}
\usepackage{titlesec}

\newtheorem{theorem}{Theorem}[section]
\newtheorem{lemma}[theorem]{Lemma}

\newtheorem{proposition}[theorem]{Proposition}
\newtheorem{conjecture}[theorem]{Conjecture}
\newtheorem{corollary}[theorem]{Corollary}
\numberwithin{equation}{section}

\theoremstyle{remark}

\theoremstyle{definition}

\titleformat{\section}
  {\normalfont\large\bfseries\centering}{\thesection}{1em}{}

\titleformat{\subsection}
  {\normalfont\bfseries}{\thesubsection}{1em}{}  

\DeclareMathOperator{\Z}{\mathbb{Z}}

\DeclareMathOperator{\mcA}{\mathcal{A}}

\DeclareMathOperator{\PP}{\mathcal{P}}
\DeclareMathOperator{\QQ}{\mathcal{Q}}

\def\reals{\hbox{\rm I\kern-.18em R}}
\def\complexes{\hbox{\rm C\kern-.43em
\vrule depth 0ex height 1.4ex width .05em\kern.41em}}
\def\field{\hbox{\rm I\kern-.18em F}} %symbol for field

\let\svthefootnote\thefootnote
\newcommand\freefootnote[1]{%
  \let\thefootnote\relax%
  \footnotetext{#1}%
  \let\thefootnote\svthefootnote%
}

\newenvironment{section*}[2][A]{
  \section*{#2}
  \renewcommand\thesection{#1}
  \setcounter{theorem}{0}}{}

\allowdisplaybreaks

\begin{document}

\title[Primes and almost primes between cubes]{Primes and almost primes between cubes}

% Oxford comma!  I don't want to be left out of this middle initial theme
\author[Johnston, Thomas, Sorenson, and Webster]{Daniel R. Johnston, Jonathan P. Sorenson, Simon N. Thomas \\and
Jonathan E. Webster}
\address{School of Science, UNSW Canberra, Australia}
\email{daniel.johnston@unsw.edu.au}
\thanks{The first author was supported by Australian Research Council Discovery Project DP240100186 and an Australian Mathematical Society Lift-off
Fellowship. The third author was supported by an Australian Government Research Training Program (RTP) Scholarship.
Computing Infrastructure at Butler University was supported by Frank Levinson and the Fairbanks Foundation.
Computing Facilities were also provided by the National Computational Infrastructure of Australia.
}
\address{Computer Science and Software Engineering Department, Butler University}
\email{jsorenso@butler.edu}
\address{School of Mathematics and Physics, The University of Queensland}
\email{s.n.thomas@student.uq.edu.au}
\address{Mathematical Sciences Department, Butler University}
\email{jewebste@butler.edu}
\date\today

\begin{abstract}
    In this paper we study the problem of detecting prime numbers between all consecutive cubes. Firstly, we use a large computation to show that there is always a prime between $n^3$ and $(n+1)^3$ for $n^3\leq 1.649\cdot 10^{40}$. In addition, we use this computation and a sieve-theoretic argument to show that there exists a number with at most 2 prime factors (counting multiplicity) between $n^3$ and $(n+1)^3$ for all $n\geq 1$. Our sieving argument uses a logarithmic weighting procedure attributed to Richert, which yields significant numerical improvements over previous approaches. 
\end{abstract}

\maketitle

\freefootnote{\textit{Corresponding author}: Simon Thomas (s.n.thomas@student.uq.edu.au).}
\freefootnote{\textit{Affiliation}: School of Mathematics and Physics, The University of Queensland.}
\freefootnote{\textit{Key phrases}: computational number theory, sieve methods, almost primes, short intervals.}
\freefootnote{\textit{2020 Mathematics Subject Classification}: 11N05, 11N36, 11Y11, 11Y16.}

\section{Introduction}

\subsection{Overview}

A longstanding open problem in number theory is Legendre's conjecture, which asks whether there is always a prime between consecutive squares $n^2$ and $(n+1)^2$. A natural weakening of this conjecture is to ask whether one can prove the existence of a prime between consecutive cubes $n^3$ and $(n+1)^3$. For large $n$, this was answered by Ingham in 1937 \cite{ingham1937difference}.

\begin{theorem}[Ingham]\label{ingthm}
    For all sufficiently large $n$, there exists a prime between $n^3$ and $(n+1)^3$.
\end{theorem}

However, despite almost 90 years passing since Ingham's result, it is still an open problem as for whether Theorem \ref{ingthm} holds for all $n\geq 1$.

\begin{conjecture}\label{cubecon}
    For all $n\geq 1$, there exists a prime between $n^3$ and $(n+1)^3$.
\end{conjecture}

Since Ingham's proof relies heavily on properties of the Riemann zeta function $\zeta(s)$, one can readily show that Conjecture \ref{cubecon} holds under assumption of the Riemann hypothesis. This was shown explicitly in 2005 by Caldwell and Cheng \cite[Lemma~5]{caldwell2005determining}, who also give an application to a prime representing function of Mills \cite{mills1947prime}.

As a first step towards resolving Conjecture \ref{cubecon} unconditionally, we consider how far one can go with modern computations. Here, state-of-the-art interval estimates \cite[Table 2]{cully2022explicit} imply that there is always a prime between $n^3$ and $(n+1)^3$ for $n^3\leq 1.42\cdot 10^{35}$. However, by adapting a recent algorithm of the second and fourth authors \cite{sorenson2025}, we are able to improve this result by a factor of $10^5$.
\begin{theorem}\label{compthm}
    There exists a prime between $n^3$ and $(n+1)^3$ for all ${n^3\leq 1.649\cdot 10^{40}}$.
\end{theorem}
In particular, the algorithm in \cite{sorenson2025} was used to find primes between consecutive squares. So, in this paper we describe how one may modify this algorithm to detect primes in more general intervals $(n^{\alpha},(n+1)^{\alpha})$, with a particular focus on the case $\alpha=3$. In total, the computation for Theorem \ref{compthm} used approximately 415,000 core hours on high-performance computing platforms, thereby giving an indication for the limits of using modern computation for this problem.

Complementary to these computations, one can then try to make Ingham's argument fully explicit to give a lower bound for the values of $n$ for which there are primes between $n^3$ and $(n+1)^3$. Unfortunately, even using modern estimates on the zeros of $\zeta(s)$, this approach appears futile. In particular, in 2016 Dudek \cite{dudek2016explicit} obtained a very large lower bound of $n\geq\exp(\exp(33.3))$, which has only been improved mildly since to $\exp(\exp(32.537))$ by Cully--Hugill in her PhD thesis \cite{cully2023explicit}.

However, modern approaches for detecting primes in intervals (see e.g.\ \cite{iwaniec1979primes,baker1996difference,baker2001difference}) often make heavy use of sieve methods, rather than the purely zeta-function estimates used by Ingham. Whilst these sieve methods are often difficult to make numerically explicit, Dudek and the first author recently had some success with using a simple sifting argument in \cite{dudek2026almost}, obtaining a weaker \emph{almost prime}\footnote{Here and throughout, by \emph{almost prime} we mean a positive integer with a small number of prime factors, counting multiplicity.} result for Legendre's conjecture. Namely, they were able to show \cite[Theorem~1.2]{dudek2026almost} that there always exists a number with at most $4$ prime factors between consecutive squares. They also state that their methods can readily be adapted to detect numbers with at most $3$ prime factors between consecutive cubes. 

In this paper, we improve the sifting argument in \cite{dudek2026almost} and utilise the computation in Theorem \ref{compthm} to obtain the following sharper result. Here and throughout, $\Omega(a)$ denotes the number of prime factors of $a$, counting multiplicity.

\begin{theorem}\label{almostthm}
    For all $n\geq 1$, there exists $a\in(n^3,(n+1)^3)$ with $\Omega(a)\leq 2$.
\end{theorem}

To prove Theorem \ref{almostthm}, we utilise a logarithmic sieve weighting procedure attributed to Richert \cite{richert1969selberg}. By contrast, in \cite{dudek2026almost} a simpler weighting procedure due to Kuhn \cite{kuhn1954neue} was used. We found, more than we expected, that Richert's logarithmic weights give much better numerical results than Kuhn's weighting procedure. For instance, we performed some rough calculations using Kuhn's weighting procedure and were not able to prove Theorem \ref{almostthm}, even for $n\geq 10^{60}$. We also note that other more complicated weighting procedures exist that improve upon Richert's weights asymptotically (see e.g.\ \cite[\S 5.2--5.3]{greaves2013sieves}). However, given the complexity of these weighting procedures, it is not clear whether they would yield better numerical results than Richert's weights, except for when dealing with very large numbers.

As an application of Theorem~1.4, one readily obtains the following almost prime variant of Mill's prime representing function\footnote{Technically, we require the slightly stronger condition $a\in(n^3,(n+1)^3-1)$ with $\Omega(a)\leq 2$. However, this still follows from our proof of Theorem~\ref{almostthm}.}.
\begin{corollary}\label{Millscor}
    Let $A=1.30637788\ldots$ be the value of Mill's constant as defined in \cite{caldwell2005determining}. Then for all $n\geq 1$, the function
    \begin{equation*}
        f(n)=\lfloor A^{3^n}\rfloor
    \end{equation*}
    has at most two prime factors counting multiplicity.
\end{corollary}

It is conjectured that $f(n)$ is prime for all $n\geq 1$ and this would essentially follow from Conjecture \ref{cubecon}. We refer the reader to \cite{elsholtz2020unconditional} for further discussion of \emph{unconditional} variants of Mill's constant. Note that we do not give the proof of Corollary \ref{Millscor} here, since the argument is identical to that in Mill's original paper \cite{mills1947prime}, combined with the computations in \cite{caldwell2005determining}.

Finally, we remark that using our approach it is not possible to improve Theorem \ref{almostthm} to $\Omega(a)=1$, even for sufficiently large $n$. Rather, to detect primes between consecutive cubes using known sieve methods, one would require a more powerful tool. Namely, one would need either an explicit version of Harman's prime-detecting sieve (see \cite[Chapter~7]{harman2007prime}) or an explicit version of Iwaniec's linear sieve with a well-factorable remainder term (see \cite[Chapter 6]{greaves2013sieves}), neither of which have been established in the literature. 

\subsection{Outline of paper}
An outline of the rest of the paper is as follows. In Section \ref{compsect} we give the computational details for Theorem \ref{compthm}. Then, in Section \ref{prelimsect}, we provide all the preliminary results required to prove Theorem \ref{almostthm}. Notably, this includes an explicit form of Richert's logarithmic sieve weighting procedure (see Section \ref{Richertsect}). Finally, in Section \ref{mainsect} we complete the proof of Theorem \ref{almostthm}.

\section{Computational details for Theorem \ref{compthm}}\label{compsect}
We begin by describing the computational details for Theorem~\ref{compthm}. Note that this section is self-contained and that the subsequent notation, chosen as to coincide with that in~\cite{sorenson2025}, should not be confused with the notation in Sections~\ref{prelimsect} and~\ref{mainsect}. 

In \cite{sorenson2025}, an algorithm was given that could find a prime in the interval $(n^2, (n+1)^2)$ for every $n < N$, thereby verifying Legendre's conjecture up to $N$.  
There were three primary components of the algorithm.
In this section, we give an overview of these components, provide a key theorem, and explain how to generalise the algorithm to higher powers with the aim of finding primes between consecutive cubes.

The first component was the Brillhart-Lehmer-Selfridge (BLS) primality test, which requires that for a candidate prime $p$, a large prime factor $R$ divides ${p-1}$~\cite{brillhart1975}.  
In this way, a proof of primality could be produced with a single modular exponentiation in most cases.  
For a candidate prime in the interval $(n^2, (n+1)^2)$, $R$ was chosen to slightly exceed $n^{2/3}$.
Since the interval is of length $2n - 1$, a given interval has around $2n^{1/3}$ numbers that are $1 \pmod{R}$ to serve as potential candidate primes.

Second, one assumes Cram\'er's model \cite{cramer1936order}, which treats the primes like a sequence of independent random variables, with any given integer having a probability of $1/\log n$ of being prime. In particular, by Cram\'er's model we only want to check about $\log(n^2)$ candidates to find a prime. One then lets $m$ be a product of small primes so that $mR \approx 2n/\log(n^2)$.
Now, under Cram\'er's model, it is highly likely that at least one candidate that is $1 \pmod{ mR }$ in the interval will be a prime.
By having many small primes in $m$, this gives an even greater chance that the candidates that are $1 \pmod {mR}$ will be prime.

Third, rather than process one interval at a time, one batch-processes a set of $t$ intervals together using the interval $(n^2, (n+t)^2)$.  
In this way, we can sieve in the arithmetic progression $1 \pmod{mR}$ to identify many candidate primes for all $t$ sub-intervals at once.  
As a result, we only attempt primality tests on candidates that are either prime or for whom all their prime factors exceed the sieve bound.
This greatly reduces the number of primality tests invoked on a given interval.  
In practice, less than $3$ primality tests were required on average per sub-interval.  

The analysis and further description of the parameter selection of $t$ may be found in \cite{sorenson2025} and we summarise by citing the key theorem.  

\begin{theorem}[{\cite[Theorem 5.1]{sorenson2025}}]
Assuming Cram\'er's model and Legendre's conjecture, there is an algorithm that can verify that Legendre's conjecture is true for all $n \leq N$ in time $O(N \log N \log \log N)$ using $O(N^{1/(\log \log N)})$ space. 
\end{theorem}

One might ask ``What role does the exponent of $2$ (the consecutive squares of Legendre's conjecture) play in the algorithm?''  
The key obstacle is the size of $R$ relative to the length of the interval.  
Consider the interval $(n^\alpha, (n+1)^\alpha)$.  
The length of the interval is approximately $\alpha n^{\alpha -1}$.  
We need $R>n^{\alpha/3}$ to use the BLS prime test.
This leaves about $\alpha n^{\alpha-1}/R < \alpha n^{(2/3)\alpha-1}$ candidates per interval, so we need $\alpha>3/2$.
We leave open what might be done if $\alpha \le 3/2$ but offer two comments. 
Since the algorithm was designed around a specific primality test, 
changing to a more costly test would require redoing the asymptotic analysis and 
the result would likely be asymptotically slower.
To keep the same primality test would require some really creative use of multiple $R$ values to guarantee enough candidates.
So, we now only consider $\alpha > 3/2$.
Fortunately for our intended application,  $3 > 2 > 3/2$, and 
we discuss the impact of choosing larger values of $\alpha$ on the three components of the algorithm.

The first step is done the same:  we choose a large candidate prime $R$ that is close to and exceeds $n^{\alpha/3}$ and 
use this prime with the BLS test.  
The second step chooses $m R \approx \alpha n^{\alpha - 1}/ \log( n^\alpha )  = n^{\alpha - 1}/ \log{n}$.
So $m \approx n^{(2\alpha - 3)/3}/\log n$.  
Note that we are making use of the strict inequality $\alpha>3/2$ to squeeze
in logarithmic factors for $m$ here.
Increasing $\alpha$ allows us to include more small primes in $m$.
In the third step, we choose $s=\alpha\log(n+t)$ (the number of candidates per interval), and thus $B$ (the sieve bound) and $t$ are basically the same as before, up to a constant.
In practice, $t$ is chosen to be as large as possible subject to everything fitting in cache.
This gives us the following:

\begin{theorem}
Let $\alpha>3/2$.
Assume that
for all integers $n>n_0(\alpha)$, there exists a prime $p$ with
$n^\alpha < p < (n+1)^\alpha$.
Assuming Cram\'er's model, there is an algorithm that can verify 
this conjecture is true
for all $n \leq N$ in time $O(N \log N \log \log N)$ using $O(N^{1/(\log \log N)})$ space. 
\end{theorem}

We implemented the above with $\alpha = 3$.  
When $\alpha = 2$ in \cite{sorenson2025}, we had $m = 30$.  
With $\alpha = 3$, we have $m = \prod_{i = 1}^{10} p_i$, the product of the first ten primes.   
Since more small primes were included in $m$, we did not use the bit-mask sieving (see Section $6$ bullet point $3$ of \cite{sorenson2025}).  
Having made these parameter and algorithmic changes, we verified that there is a prime between consecutive cubes for all intervals $(n^3, (n+1)^3)$ with $n\leq 2.54569\cdot 10^{13}$ and thus $n^3\leq 1.649\cdot 10^{40}$. That is, we obtained Theorem~\ref{compthm}.

The computation was completed on a number of different high-performance computing platforms.
At Butler University, we again used the four Intel Xeon Phi 7210 processors (256 cores) and the 192-core cluster of Intel Xeon E5-2630 2.3 GHz processors, as was done for the previous computation from \cite{sorenson2025}.
In addition, we used a new cluster of 4 nodes/448 cores of
AMD EPYC 9734 2.2GHz processors at Butler,
and NCI's Gadi supercomputer in Australia.
We used about 62k CPU hours on Gadi, one month of wall time each on the 256 Phi cores and on the 192-core cluster, and 3 days of wall time on the new 448 AMD cores.
This totals roughly 415k core hours of work.

Our code and data can be found here:
\url{https://github.com/sorenson64/cubes}

\section{Sieve-theoretic preliminaries}\label{prelimsect}

We now move on to the preliminaries for the proof of Theorem~\ref{almostthm}. As is customary in sieve-theoretic problems, we will specify a finite set of positive integers $\mcA$, and an infinite set of primes $\PP$, and proceed to ``sift out'' the elements of $\mcA$ with prime divisors in $\PP$ which are less than some number $z$. To this end, we let $\PP$ be the set of all primes and 
\begin{equation}\label{Adef}
  \mcA=\mcA(N):=\Z\cap (N, N+3N^{2/3}).
\end{equation}
Upon setting $N=n^3$, one sees that $\mcA$ is contained in the interval $(n^3, (n+1)^3)$. We note that one could also define $\mcA$ to be the slightly larger set 
\begin{equation*}
  \Z\cap (N, N+3N^{2/3}+3N^{1/3}+1),
\end{equation*}
but the simpler \eqref{Adef} will be sufficient for the large values of $N$ we deal with.

The number of elements of $\mcA$ which survive the sifting by $\PP$ is given by the sifting function 
\begin{equation}\label{sfunctiondef}
  S(\mcA, \PP, z):=\left|\mcA\setminus \bigcup_{p|P(z)}\mcA_p\right|,
\end{equation}
where $p$ is prime, %$z\ge 2$,
\begin{equation}\label{pzdef}
  \mcA_p:=\{a\in \mcA:p|a\},\quad \text{and}\quad P(z):=\prod_{\substack{p<z\\p\in \PP}}p.
\end{equation}

\subsection{An explicit linear sieve}

To prove the existence of almost primes in $\mcA$, we require upper and lower bounds for $S(\mcA, \PP, z)$. These will come from the explicit linear sieve due to Bordignon, Starichova, and the first author \cite{BJV24}. Here, we give a simplified statement of this result, also used in \cite{dudek2026almost}. % consider rewording to differ from dj and adrian's paper.

\begin{lemma}[{Explicit version of the linear sieve \cite[Theorem~2.1]{BJV24}}]\label{linearsieve} %maybe this reference should be to the Bordignon paper?
  Let $\mcA$ be a finite set of positive integers and $\PP$ be an infinite set of primes such that no element of $\mcA$ is divisible by a prime in the complement $\overline{\PP}=\{p\text{ prime}:p\notin \PP\}$. Let $P(z)$ and $S(\mcA, \PP, z)$ be as in \eqref{pzdef} and \eqref{sfunctiondef} respectively. For any square-free integer $d$, let $g(d)$ be a multiplicative function with 
  \begin{equation*}
    0\le g(p)<1\text{ for all }p\in \PP,
  \end{equation*}
  and 
  \begin{equation*}
    \mcA_d:=\{a\in \mcA:d|a\}\quad \text{and}\quad r(d)=|\mcA_d|-|\mcA|g(d).
  \end{equation*}
  Let $\QQ\subset \PP$, and $Q=\prod_{q\in \QQ}q$. Suppose that, for some $\varepsilon$ satisfying $0<\varepsilon\le 1/74$, the inequality 
  \begin{equation}\label{epcondition}
    \prod_{\substack{p\in \PP\setminus \QQ\\ u\le p<z}}(1-g(p))^{-1}<(1+\varepsilon)\frac{\log z}{\log u}
  \end{equation}
  holds for all $1<u<z$. Then, for any $D\ge z$ we have the upper bound
  \begin{equation}\label{sieveupper}
    S(\mcA, \PP, z)<|\mcA|V(z)\cdot (F(s)+\varepsilon C_1(\varepsilon)e^2h(s))+R(D)
  \end{equation}
  and for any $D\ge z^2$ we have the lower bound 
  \begin{equation}\label{sievelower}
    S(\mcA, \PP, z)>|\mcA|V(z)\cdot (f(s)-\varepsilon C_2(\varepsilon)e^2h(s))-R(D),
  \end{equation}
  where
  \begin{equation*}
    s=\frac{\log D}{\log z},
  \end{equation*}
  \begin{equation}\label{hsdef}
    h(s):=\begin{cases}
      s^{-1}e^{-2},&0< s\leq 1,\\
      e^{-2}&1\le s\le 2,\\
      e^{-s}&2\le s\le 3,\\
      3s^{-1}e^{-s}& s\ge 3,
    \end{cases}
  \end{equation}
  $F(s)$ and $f(s)$ are the two functions defined by the following delay differential equations:
  \begin{equation}\label{Ffdef}
    \begin{cases}
      F(s)=\frac{2e^{\gamma}}{s},\quad f(s)=0 & \text{for } 0<s\le 2,\\
      (sF(s))'=f(s-1),\quad (sf(s))'=F(s-1)&\text{for }s\ge 2,
    \end{cases}
  \end{equation}
  $C_1(\varepsilon)$ and $C_2(\varepsilon)$ come from \cite[Table 1]{BJV24},
  \begin{equation*}
    V(z):=\prod_{p|P(z)}(1-g(p)),
  \end{equation*}
  and the remainder term is
  \begin{equation*}
    R(D):=\sum_{\substack{d|P(z)\\ d<QD}}|r(d)|.
  \end{equation*}
\end{lemma}

From \eqref{Ffdef}, we have the following expressions for $F(s)$ when $0< s\leq 3$ and $f(s)$ when $2\le s\le 4$, which are sufficient ranges of $s$ to consider for our application.

\begin{lemma}\label{fFlem}
  Let $F(s)$ and $f(s)$ be as in \eqref{Ffdef}. Then for $0< s\leq 3$,
  \begin{equation*}
      F(s)=\frac{2e^{\gamma}}{s}
  \end{equation*}
  and for $2\leq s\leq 4$,
  \begin{equation*}
    f(s)=\frac{2e^\gamma\log(s-1)}{s}.
  \end{equation*}
\end{lemma}

We remark that the definition for $h(s)$ when $0<s\leq 1$ in \eqref{hsdef} is nonstandard and in fact does not appear in the statement of \cite[Theorem~2.1]{BJV24}. This is because this range of $s$ is outside the limits of where the sieve upper and lower bounds, \eqref{sieveupper} and \eqref{sievelower}, can be applied. However, setting $h(s)=s^{-1}e^{-2}$ for $0< s\leq 1$ is convenient notation for our application, since we will require a (weak) upper bound for $S(\mcA,\PP,z)$ that essentially extends \eqref{sieveupper} to the case $D<z$.

The \emph{density function} $g(d)$ in Lemma \ref{linearsieve} should be chosen so that $|\mcA_d|\approx g(d)|\mcA|$, yielding a small value for $|r(d)|$. Since for our application, $\mcA$ is just a set of consecutive integers, we choose $g(d)=\frac{1}{d}$ which gives
\begin{equation*}
  r(d)=|\mcA_d|-\frac{|\mcA|}{d}.
\end{equation*}
With this choice of $g(d)$, we have
\begin{equation*}
    V(z)=\prod_{p<z}\left(1-\frac{1}{p}\right)
\end{equation*}
for which one has the following simple bounds.
\begin{lemma}[{\cite[Theorem~7]{rosser1962approximate}}]\label{rosserlem}
    For all $z>1$,
    \begin{equation*}
        \prod_{p<z}\left(1-\frac{1}{p}\right)\leq \frac{e^{-\gamma}}{\log z}\left(1+\frac{1}{2(\log z)^2}\right)
    \end{equation*}
    and for all $z\geq 285$,
    \begin{equation*}
        \prod_{p< z}\left(1-\frac{1}{p}\right)\geq\frac{e^{-\gamma}}{\log z}\left(1-\frac{1}{2(\log z)^2}\right).
    \end{equation*}
\end{lemma}

Before continuing, we give an explicit version of \eqref{epcondition} with $g(p)=1/p$ and a range of $z$ that suffices for our application.

\begin{lemma}\label{PrimeProductBound}
    For all $u\ge 3$ and $z\ge 10^5$, we have 
    \begin{equation}\label{epsbound2}
        \prod_{u\le p<z}\left(1-\frac{1}{p}\right)^{-1}<(1+2.8\cdot 10^{-4})\frac{\log z}{\log u}.
    \end{equation}
\end{lemma}
\begin{proof}
    The proof of~\eqref{epsbound2} is very similar to~\cite[Proposition~A.2]{dudek2026almost}. We thereby omit some details for brevity, referring the reader to the appendix of \cite{dudek2026almost}.

    For the proof, we consider three cases for $z$.\\
    
    \noindent\textbf{Case 1:} $10^5\le z<3.5\cdot 10^5$.\\
    \noindent For this range of $z$, we verified \eqref{epsbound2} for all $3\leq u\leq z$ by a direct computation. This took just under 9 hours on a laptop with a 1.9 GHz processor. A short Python script for this code is available at \url{https://github.com/DJmath1729/MertenBounds}.\\
    \\
    \textbf{Case 2:} $3.5\cdot 10^5\le z<4\cdot 10^9$.\\
    For this range of $z$, we apply \cite[(A.1)]{dudek2026almost}, which states that
    \begin{align}\label{mert1}
        e^\gamma\log x<\prod_{p\leq x}\left(1-\frac{1}{p}\right)^{-1}<e^\gamma\log x+\frac{2e^{\gamma}}{\sqrt{x}},
    \end{align}
    for all $2\leq x\leq 4\cdot 10^9$. In particular, for $3.5\cdot 10^5\leq z<4\cdot 10^9$, \eqref{mert1} gives
    \begin{equation*}
        \prod_{u\le p<z}\left(1-\frac{1}{p}\right)^{-1}\leq \frac{\log z}{\log u}\left(1+\frac{2}{\sqrt{z}\log z}\right)<\frac{\log z}{\log u}\left(1+2.7\cdot 10^{-4}\right),
    \end{equation*}
    which is a little stronger than \eqref{epsbound2}.\\ \\
    \textbf{Case 3:} $z\ge 4\cdot 10^9$.\\ 
    This case is identical to Case 3 in the proof of \cite[Proposition~A.2]{dudek2026almost}, whereby \eqref{epsbound2} follows from \eqref{mert1} and further bounds on $\prod_{p\leq x}\left(1-1/p\right)^{-1}$ due to Vanlalngaia~\cite[Theorem~5]{vanlalngaia2017explicit}.
\end{proof}
In the notation of Lemma~\ref{linearsieve}, the inequality \eqref{epsbound2} says that one may take $Q=2$ and $\varepsilon=2.8\cdot 10^{-4}$ when $z\geq 10^5$.

\subsection{Richert's weighted sieve}\label{Richertsect}

We now give an explicit version of Richert's logarithmic sieve-weighting procedure, which first appeared in \cite{richert1969selberg}. This will allow us to obtain a numerically good lower bound for the function
\begin{equation}\label{rkdef}
    r_k(\mcA):=|\{a\in \mcA:\Omega(a)\le k\}|
\end{equation}
in terms of the sifting function $S(\mcA,\PP,z)$. Note that throughout this subsection, each result is expressed generally, so that they in fact hold for any finite set of integers $\mcA$, rather than just our choice of $\mcA$ in \eqref{Adef}. 

In what follows, let $k_1$ and $k_2$ be real numbers with $k_1\ge k_2\ge 1$ and set 
\begin{equation}\label{Xdef}
  X:=\max(\mcA),% make consistent with m0 condition
\end{equation}
\begin{equation}\label{yzdef}
  z=X^{1/k_1},\quad y=X^{1/k_2}.
\end{equation}

Next, let $k$ be a natural number subject to the restriction $k+1>k_2$. For our application, we set $k=2$, which is indicative of the fact that we are trying to detect $a\in A$ with $\Omega(a)\leq 2$. % consider rewording

% I'm not sure that we need this next condition.
% and for some constant $A_1\ge 1$,
% \begin{equation*}
%   \sum_{d<X^\tau(\log X)^{-A_1}}\mu^2(d)4^{\nu(d)}|r(d)|=O\left(\frac{X}{(\log X)^{2}}\right)
% \end{equation*}

For any $a\in \mcA$ with $(a, P(z))=1$, Richert's logarithmic weights are given by 
\begin{equation}\label{Richertweight}
  w(a)=\lambda - \sum_{\substack{z\le p<y\\p\in \PP,\thinspace p\mid a}}\left(1-\frac{\log p}{\log y}\right),
\end{equation}
where $\lambda$ is a parameter which we set to be
\begin{equation}\label{lambdadef}
  \lambda:=k+1-k_2.
\end{equation}

We then define a weighted sifting function $W(\mcA, \PP, z)$ by 
\begin{equation}\label{Wdef}
  W(\mcA,\PP,z)=W(\mcA, \PP, z, k_1, k_2, \lambda):=\sum_{\substack{a\in \mcA\\ (a, P(z))=1}}w(a).
\end{equation}% should this have $z$ as one of the arguments?
Our first goal is to relate $W(\mcA,\PP,z)$ to $r_k(\mcA)$. This is achieved by considering a subset $\mcA'$ of $\mcA$, defined as
\begin{equation}\label{Aprimedef}
    \mcA'=\mcA\setminus \bigcup_{\substack{z\le p<y\\ p\in \PP}}\mcA_{p^2}.
\end{equation}
Here, $\mcA'$ does not contain any elements of $\mcA$ with ``moderately large" square divisors, and for most applications
\begin{equation*}
    \sum_{\substack{z\leq p<y\\ p\in\mathcal{P}}}|\mcA_{p^2}|
\end{equation*}
is negligible compared to $\mcA$. 

\begin{lemma}\label{rklemma1}
   Let $\mcA'$ be as in \eqref{Aprimedef}. Then 
   \begin{equation*}
    r_k(\mcA) \ge \frac{1}{k}W(\mcA', \PP,z),
  \end{equation*}
  where $r_k(\mcA)$ and $W(\mcA',\PP,z)$ are as defined in \eqref{rkdef} and \eqref{Wdef} respectively. % Probably better to define $r_k$ and $\Omega$ prior to this.
\end{lemma}
\begin{proof}
  Let $a'\in \mcA'$ with $(a', P(z))=1$. For any prime factor $p$ of $a'$, if $p^2|a'$, then one must have $p\ge y$ and therefore $1-\frac{\log p}{\log y}\le 0$. In particular,
  \begin{align}\label{wbound1}
    \sum_{\substack{p\ge z,\thinspace \alpha\geq 1\\ p\in \PP,\thinspace p^\alpha\mid a'}}\left(1-\frac{\log p}{\log y}\right)&=\sum_{\substack{z\le p<y\\ p\in \PP,\thinspace p\mid a'}}\left(1-\frac{\log p}{\log y}\right)+\sum_{\substack{p\ge y,\thinspace \alpha\geq 1\\p\in \PP,\thinspace p^\alpha\mid a'}}\left(1-\frac{\log p}{\log y}\right)\\
    &\le \sum_{\substack{z\le p<y\\p\in \PP,\thinspace p\mid a'}}\left(1-\frac{\log p}{\log y}\right).
  \end{align}
  From here, it follows that
  \begin{equation*}
    w(a')\le \lambda-\sum_{\substack{p\ge z,\thinspace \alpha\geq 1\\ p\in \PP,\thinspace p^\alpha\mid a'}}\left(1-\frac{\log p}{\log y}\right)=\lambda-\Omega(a')+\frac{k_2\log a'}{\log X}.
  \end{equation*}
  By the definition \eqref{Xdef} of $X$ and the choice \eqref{lambdadef} of $\lambda$, one then has
  \begin{align}\label{wapfirsteq}
    W(\mcA', \PP, z)&\le \sum_{\substack{a'\in \mcA'\\ (a', P(z))=1}}\left(\lambda-\Omega(a')+\frac{k_2\log a'}{\log X}\right)\notag\\
    &\le \sum_{\substack{a'\in \mcA'\\ (a', P(z))=1}}\left(\lambda-\Omega(a')+k_2\right)\notag\\
    &=\sum_{\substack{a'\in \mcA'\\ (a', P(z))=1}}\left(k+1-\Omega(a')\right).
  \end{align}
  Now, if $\Omega(a')>k$, the summand in \eqref{wapfirsteq} is non-positive. Then, if $\Omega(a')\leq k$, the summand is at most $k$. Hence
  \begin{equation*}
    W(\mcA', \PP, z)\le k\sum_{\substack{a'\in \mcA'\\ (a', P(z))=1\\ \Omega(a')\le k}}1,
  \end{equation*}
  and in particular,
  \begin{equation*}
    r_k(\mcA)\ge r_k(\mcA')\ge \frac{1}{k}W(\mcA', \PP,z).\qedhere
  \end{equation*}
\end{proof}

Our next goal is to relate $W(\mcA',\PP,z)$ to $S(\mcA,\PP,z)$, so that $r_k(\mcA)$ may be bounded using Lemmas~\ref{linearsieve} and~\ref{rklemma1}. In what follows, we will assume the condition
\begin{equation}\tag{$Q_0(c,\delta)$}\label{q0cond}
    \sum_{\substack{z\le p<y\\ p\in \PP}}|\mcA_{p^2}|\leq c|\mcA|^{1-\delta},
\end{equation}
for some constants $c>0$ and $0<\delta<1$. This implies that $|\mcA\setminus\mcA'|$ is small.

\begin{lemma}\label{rklemma2} % These two lemmas should potentially be combined. Also, define sifting function.
  Assuming \eqref{q0cond}, we have 
  \begin{align}\label{rklowermain}
    &r_k(\mcA)\ge \frac{\lambda}{k}S(\mcA, \PP, z)-\frac{1}{k}\sum_{\substack{p\in \PP\\ z\le p<y}}\left(1-\frac{\log p}{\log y}\right) S(\mcA_p, \PP, z)-\frac{\lambda c}{k}|\mcA|^{1-\delta},
  \end{align}
  where $\lambda$ is as defined in \eqref{lambdadef}.
\end{lemma}

\begin{proof}
  From Lemma \ref{rklemma1},
  \begin{align}\label{rklower1}
    r_k(\mcA)&\ge \frac{1}{k}W(\mcA', \PP, z)\notag\\
    &\ge\frac{1}{k}\left(\sum_{\substack{a\in \mcA\\(a,P(z))=1}}w(a)-\sum_{\substack{z\le p<y\\p\in \PP}}\sum_{\substack{a\in \mcA_{p^2}\\ (a, P(z))=1}}w(a)\right).
  \end{align}

    Since $w(a)\leq\lambda$,  the second term in \eqref{rklower1} can be bounded using \eqref{q0cond} by
    \begin{align}\label{rkfirst}
        \sum_{\substack{z\le p<y\\p\in \PP}}\sum_{\substack{a\in \mcA_{p^2}\\ (a, P(z))=1}}w(a)&\le \lambda\sum_{\substack{z\le p<y\\p\in \PP}}|\mcA_{p^2}|\notag\\
        &\le \lambda c|\mcA|^{1-\delta}.
    \end{align}

    For the first term in \eqref{rklower1}, we then have
    \begin{align}\label{rksecond}
      \sum_{\substack{a\in \mcA\\(a, P(z))=1}}w(a)&=\sum_{\substack{a\in \mcA\\(a, P(z))=1}}\left(\lambda-\sum_{\substack{p\in \PP\\ p|a\\ z\le p<y}}\left(1-\frac{\log p}{\log y}\right)\right)\notag\\
      &=\lambda S(\mcA, \PP, z)-\sum_{\substack{p\in \PP\\ z\le p<y}}\left(1-\frac{\log p}{\log y}\right)S(\mcA_p, \PP, z).
  \end{align}
  Substituting \eqref{rkfirst} and \eqref{rksecond} into \eqref{rklower1} yields the claimed bound for $r_k(\mcA)$.
\end{proof}

\subsection{An upper bound for the sum over \texorpdfstring{$S(\mcA_p,\PP,z)$}{}}
From Lemma \ref{rklemma2}, it appears that we are almost in a position to bound $r_k(\mcA)$ via an application of the explicit linear sieve (Lemma~\ref{linearsieve}). However, care is still required to explicitly bound the sum
\begin{equation}\label{Sapsum}
    \sum_{\substack{p\in \PP\\ z\le p<y}}\left(1-\frac{\log p}{\log y}\right) S(\mcA_p, \PP, z).
\end{equation}
This will be the goal of the rest of this section. We begin with a couple of lemmas from the literature.

\begin{lemma}[{\cite[\S1.3.5, Lemma 1 (ii)]{greaves2013sieves}}]\label{greaveslemma}
  Let $f(t)$ be a positive monotone function defined for $z\le t\le y$ with $f'(t)$ piecewise continuous on $[z,y]$, and $c(n)$ be an arithmetic function satisfying 
  \begin{equation*}
    \sum_{x\le n<w}c(n)\le g(w)-g(x)+E,
  \end{equation*}
  for some constant $E$ whenever $z\le x<w\le y$. Then,
  \begin{equation*}
    \sum_{z\le n<y}c(n)f(n)\le \int_z^y f(t)g'(t)\mathrm{d}t+E\max (f(z),f(y)).
  \end{equation*}
\end{lemma}

\begin{lemma}[{\cite[Corollary 1]{vanlalngaia2017explicit}}]\label{vanlalngaia2017explicit}
  For any $b>a>1000$,
  \begin{equation*}
    \sum_{a\le p<b}\frac{1}{p}<\log \log b-\log \log a+\frac{5}{\left(\log a\right)^3}.
  \end{equation*}
\end{lemma}

We now prove the key proposition, which gives an upper bound for \eqref{rklowermain} with our choice \eqref{Adef} of $\mcA$. This can be directly compared to \cite[Proposition~3.5]{dudek2026almost}, where a similar upper bound is obtained in the context of Kuhn's weighted sieve, which is simpler yet less optimal than our use of Richert's weighted sieve. Compared to the argument in \cite{dudek2026almost} we have also weakened the restriction on $\alpha$ (see \eqref{alphabounds}), which gives a greater flexibility in parameters, but requires us to carefully consider different ranges of $q$ in the sum \eqref{Sapsum}.

\begin{proposition}\label{newprop35-2}
    Let $\mcA$ be as defined in \eqref{Adef}, $\mathcal{P}$ be the set of all primes, and $X,y,z$ be as in \eqref{Xdef} and \eqref{yzdef}. Suppose $y>z>1000$ and $k_1\le 6$. Let $\alpha$ be a real number such that 
    \begin{equation}\label{alphabounds}
        0<\alpha<\frac{2}{3}-\frac{1}{k_2}
    \end{equation}
    and set
    \begin{equation*}
        k_\alpha=k_1\left(\frac{2}{3}-\frac{1}{k_2}-\alpha\right).
    \end{equation*}
    Also, let 
    \begin{equation*}
        D_q:=\frac{X^{\frac{2}{3}-\alpha}}{q},\quad s_q=\frac{\log D_q}{\log z},\quad \widetilde{D}=\min\{D_y,z\},
    \end{equation*}
    $\mathcal{Q}$ be a set of primes, $Q=\prod_{q\in\mathcal{Q}}q$ and $\varepsilon>0$ be such that
    \begin{equation*}
        \prod_{\substack{p\in \PP\setminus \QQ\\ u\le p<x}}\left(1-\frac{1}{p}\right)^{-1}<(1+\varepsilon)\frac{\log x}{\log u}
    \end{equation*}
    for all $x\geq\widetilde{D}$ and $1<u\leq x$.
    Then,
    \begin{equation}\label{sumupper}
        \sum_{\substack{q\in \PP\\ z\le q<y}}\left(1-\frac{\log q}{\log y}\right)S(\mcA_q, \PP, z)< k_1 e^{-\gamma}\left(1+\frac{1}{2(\log \widetilde{D})^2}\right)(M_1(X)+M_2(X))+\mathcal{E}(X),
    \end{equation}
    where 
    \begin{align}
        &M_1(X)\notag\\
        &:=\frac{3N^{2/3}}{\log X}\left[\frac{2e^{\gamma}}{k_1}\left(\frac{1}{(2-3\alpha)}\left(3\log \frac{k_1}{k_2}+(k_2(2-3\alpha)-3)\log \frac{2-3\alpha-\frac{3}{k_2}}{2-3\alpha-\frac{3}{k_1}}\right)\right.\right.\notag\\
        &\left.+\frac{5k_1^4}{k_\alpha(\log X)^3}\left(1-\frac{k_2}{k_1}\right)\right)
        \left.+\varepsilon C_1(\varepsilon)e^2h(k_\alpha)\left(\log \frac{k_1}{k_2}-1+\frac{k_2}{k_1}+\frac{5k_1^3}{(\log X)^3}\left(1-\frac{k_2}{k_1}\right)\right)\right],\label{M1def}\\
        &M_2(X):=\frac{y}{\log X}\left(1-\frac{k_2}{k_1}\right)\left(\frac{2e^\gamma}{k_\alpha}+\varepsilon C_1(\varepsilon)e^2h(k_\alpha)\right),\label{M2def}\\
        &\mathcal{E}(X):=QX^{\frac{2}{3}-\alpha}\left(\log \frac{k_1}{k_2}-1+\frac{k_2}{k_1}+\frac{5k_1^3}{\left(\log X\right)^3}\left(1-\frac{k_2}{k_1}\right)\right).\label{calEdef}
    \end{align}
    with $C_1(\varepsilon)$ and $h(s)$ as defined in Lemma~\ref{linearsieve}.
\end{proposition}

\begin{proof}
    We begin by giving an upper bound for $S(\mcA_q,\PP,z)$ that holds for all ${z\leq q<y}$. For this, we let
    \begin{equation*}
        w:=X^{\frac{2}{3}-\alpha-\frac{1}{k_1}}
    \end{equation*}
    and separately consider the cases $z\leq q\leq w$ and $w<q<y$.\\ \\
    \textbf{Case 1: $z\leq q\leq w$}\\
    In this case, $D_q\geq z$ so that we may directly apply \eqref{sieveupper} of Lemma~\ref{linearsieve} with $g(d)=1/d$, $D=D_q$ and $s=s_q$ to give
    \begin{align}\label{ub1}
        S(\mcA_q, \PP, z)<
        |\mcA_q|V(z)(F(s_q)+\varepsilon C_1(\varepsilon)e^2h(s_q))+\sum_{\substack{d\mid P(z)\\d<QD_q}}|r_q(d)|,
    \end{align}
    where
    \begin{equation}\label{rqdef}
        r_q(d):=|\mcA_{qd}|-\frac{|\mcA_q|}{d}
    \end{equation}
    and $F(s_q)$ is as defined in \eqref{Ffdef}. Now, by Lemma~\ref{rosserlem},
    \begin{equation*}
      V(z)=\prod_{p\le z}\left(1-\frac{1}{p}\right)\leq\frac{e^{-\gamma}}{\log z}\left(1+\frac{1}{2(\log z)^2}\right)\leq\frac{k_1e^{-\gamma}}{\log X}\left(1+\frac{1}{2(\log \widetilde{D})^2}\right),
    \end{equation*}
    so that \eqref{ub1} can be bounded above further by
    \begin{multline}\label{ub2}
        k_1e^{-\gamma}\left(1+\frac{1}{2(\log \widetilde{D})^2}\right)|\mcA_q|\left(\frac{F(s_q)+\varepsilon C_1(\varepsilon)e^2h(s_q)}{\log X}\right)+\sum_{\substack{d\mid P(z)\\d<QD_q}}|r_q(d)|.
    \end{multline}
    \noindent\textbf{Case 2: $w< q< y$}\\
    In this case, since $q>w$ we have $z>D_q$. So, we can apply the simple bound
    \begin{equation*}
      S(\mcA_q, \PP, z)\le S(\mcA_q, \PP, D_q),
    \end{equation*}
    noting that $S(\mcA_q,\PP,z)$ sifts out more elements of $\mcA$ than $S(\mcA_q,\PP,D_q)$.

    Therefore, applying \eqref{sieveupper} of Lemma~\ref{linearsieve} with $D=D_q$ and $s=1$, we obtain
    \begin{align}\label{SAqDqfirst}
        S(\mcA_q,\PP,z)&\leq|\mcA_q|V(D_q)(F(1)+\varepsilon C_1(\varepsilon)e^2h(1))+\sum_{\substack{d\mid P(z)\\d<QD_q}}|r_q(d)|,\notag\\
        &=|\mcA_q|V(D_q)(2e^{\gamma}+\varepsilon C_1(\varepsilon))+\sum_{\substack{d\mid P(z)\\d<QD_q}}|r_q(d)|
    \end{align}
    with $r_q(d)$ again as defined in \eqref{rqdef}. Here, Lemma~\ref{rosserlem} gives,
    \begin{equation}\label{Vdqupper}
        V(D_q)\leq\frac{e^{-\gamma}}{\log D_q}\left(1+\frac{1}{2(\log D_q)^2}\right)\leq\frac{k_1e^{-\gamma}}{s_q\log X}\left(1+\frac{1}{2(\log \widetilde{D})^2}\right).
    \end{equation}
    Substituting \eqref{Vdqupper} into \eqref{SAqDqfirst} and using that $D_q\ge D_y$,
    \begin{align}\label{largequb}
      &S(\mcA_q, \PP, z)\notag\\
      &\le \frac{k_1e^{-\gamma}}{\log X}\left(1+\frac{1}{2(\log D_y)^2}\right)|\mcA_q|\left(\frac{2e^{\gamma}}{s_q}+\frac{\varepsilon C_1(\varepsilon)}{s_q}\right)+\sum_{\substack{d\mid P(z)\\d<QD_q}}|r_q(d)|\notag\\
      &= k_1e^{-\gamma}\left(1+\frac{1}{2(\log \widetilde{D})^2}\right)|\mcA_q|\left(\frac{F(s_q)+\varepsilon C_1(\varepsilon)e^2h(s_q)}{\log X}\right)+\sum_{\substack{d\mid P(z)\\d<QD_q}}|r_q(d)|.
    \end{align}
    \ \\
    Notably, in both cases above we have obtained the same upper bound for $S(\mcA_q,\PP,z)$, \eqref{ub2} and \eqref{largequb}. Hence, we can write
    \begin{align}\label{bigsapeq}
      &\sum_{\substack{q\in \PP\\ z\le q<y}}\left(1-\frac{\log q}{\log y}\right)S(\mcA_q, \PP, z)\notag\\
      &\le k_1e^{-\gamma}\left(1+\frac{1}{2(\log \widetilde{D})^2}\right)\sum_{\substack{q\in \PP\\ z\le q < y}}\left(1-\frac{\log q}{\log y}\right)|\mcA_q|\left(\frac{F(s_q)+\varepsilon C_1(\varepsilon)e^2h(s_q)}{{\log X}}\right)\notag\\
      &\qquad\qquad\qquad\qquad\qquad\qquad\qquad\qquad+\sum_{\substack{q\in \PP\\ z\le q < y}}\left(1-\frac{\log q}{\log y}\right)\sum_{\substack{d\mid P(z)\\d<QD_q}}|r_q(d)|.
    \end{align}
    To bound the remainder term in \eqref{bigsapeq}, we note that since $\mcA$ is a set of consecutive integers, one has 
    \begin{equation*}
      |r_q(d)|=\left||\mcA_{qd}|-\frac{|\mcA_q|}{d}\right|\le 1.
    \end{equation*}
    Therefore,
    \begin{equation*}
        \sum_{\substack{z\le q<y\\q\in \mathcal{P}}}\left(1-\frac{\log q}{\log y}\right)\sum_{\substack{d\mid P(z)\\d<QD_q}}|r_q(d)|\le QX^{\frac{2}{3}-\alpha}\sum_{\substack{z\le q<y\\q\in \mathcal{P}}}\left(1-\frac{\log q}{\log y}\right)\frac{1}{q}.
    \end{equation*}
    We can now apply Lemmas~\ref{greaveslemma} and \ref{vanlalngaia2017explicit}. In particular, in the notation of Lemma~\ref{greaveslemma}, setting $g(x)=\log\log x$, $E=\frac{5k_1^3}{(\log X)^3}$,
    \begin{equation*}
        c(n)=\begin{cases}
        \frac{1}{n}&n\text{ prime}\\
        0&\text{ else}
        \end{cases},
    \end{equation*}
    and $f(t)=1-\frac{\log t}{\log y}$, yields
    \begin{equation}\label{ubl32}
        \sum_{\substack{z\le q<y\\q\in \mathcal{P}}}\left(1-\frac{\log q}{\log y}\right)\frac{1}{q}\le \log \frac{k_1}{k_2}-1+\frac{k_2}{k_1}+\frac{5k_1^3}{(\log X)^3}\left(1-\frac{k_2}{k_1}\right),
    \end{equation}
    with the value of $E=5k_1^3/(\log X)^3$ arising from Lemma~\ref{vanlalngaia2017explicit}.
    Thus, 
    \begin{align*}
        &\sum_{\substack{z\le q<y\\q\in \mathcal{P}}}\left(1-\frac{\log q}{\log y}\right)\sum_{\substack{d\mid P(z)\\d<QD_q}}|r_q(d)|\\
        &\qquad\le QX^{\frac{2}{3}-\alpha}\left(\log \frac{k_1}{k_2}-1+\frac{k_2}{k_1}+\frac{5k_1^3}{\left(\log X\right)^3}\left(1-\frac{k_2}{k_1}\right)\right)=\mathcal{E}(X).
    \end{align*}

    Next we deal with the main term in \eqref{bigsapeq}. Here, we begin by noting that
    \begin{equation}\label{Aqsimp}
        |\mcA_q|\le \frac{3N^{2/3}}{q}+1
    \end{equation}
    since $\mcA$ is a set of consecutive integers. Now, because $k_1\le 6$, we have $0<s_q\le 3$ and thus by Lemma~\ref{fFlem},
    \begin{equation}\label{Fsqsimp}
        F(s_q)=\frac{2e^\gamma}{s_q}=\frac{2e^\gamma\log X}{k_1\log D_q}.
    \end{equation}

    Using \eqref{Aqsimp}, \eqref{Fsqsimp}, the bounds $s_q\ge k_\alpha$ and $D_q\ge D_y$, and the fact that $h(s)$ is non-increasing, we then get
    \begin{multline}\label{ubm1}
        \sum_{\substack{q\in \PP\\ z\le q<y}}\left(1-\frac{\log q}{\log y}\right)|\mcA_q|\left(\frac{F(s_q)+\varepsilon C_1(\varepsilon)e^2h(s_q)}{\log X}\right)\\
        \le 3N^{2/3}\sum_{z\le q<y}\frac{1}{q}\left(1-k_2\frac{\log q}{\log X}\right)\left(\frac{2e^\gamma}{k_1\log D_q}+\frac{\varepsilon C_1(\varepsilon)e^2 h(k_\alpha)}{\log X}\right)\\
        +\sum_{z\le q<y}\left(1-\frac{k_2}{k_1}\right)\left(\frac{2e^\gamma}{k_1\log D_y}+\frac{\varepsilon C_1(\varepsilon)e^2 h(k_\alpha)}{\log X}\right).
    \end{multline}

    For most choices of parameters, the second term in \eqref{ubm1} is negligible and can be bounded as
    \begin{multline*}
        \sum_{z\le q<y}\left(1-\frac{k_2}{k_1}\right)\left(\frac{2e^\gamma}{k_1\log D_y}+\frac{\varepsilon C_1(\varepsilon)e^2 h(k_\alpha)}{\log X}\right)\\
        \le \frac{y}{\log X}\left(1-\frac{k_2}{k_1}\right)\left(\frac{2e^\gamma}{k_\alpha}+\varepsilon C_1(\varepsilon)e^2h(k_\alpha)\right)=M_2(X).
    \end{multline*}

    Finally, for the first term in \eqref{ubm1}, we apply Lemmas~\ref{greaveslemma} and \ref{vanlalngaia2017explicit} as before (see \eqref{ubl32}) to obtain
    \begin{align*}
        &\sum_{z\le q<y}\left(1-k_2\frac{\log q}{\log X}\right)\frac{1}{q\log D_q}\\
        &\qquad\le \frac{1}{(2-3\alpha)\log X}\left(3\log \frac{k_1}{k_2}+(k_2(2-3\alpha)-3)\log \frac{2-3\alpha-\frac{3}{k_2}}{2-3\alpha-\frac{3}{k_1}}\right)\\
        &\qquad\qquad\qquad\qquad\qquad\qquad\qquad\qquad\qquad\qquad\qquad+\frac{5k_1^4}{k_\alpha(\log X)^4}\left(1-\frac{k_2}{k_1}\right).
    \end{align*}
    When combined with \eqref{ubl32} and \eqref{ubm1}, this gives
    \begin{equation*}
        \sum_{\substack{q\in \PP\\ z\le q<y}}\left(1-\frac{\log q}{\log y}\right)|\mcA_q|\left(\frac{F(s_q)+\varepsilon C_1(\varepsilon)e^2h(s_q)}{\log X}\right)\le M_1(X)+M_2(X),
    \end{equation*}
    which completes the proof of the proposition.
\end{proof}

\section{Proof of Theorem \ref{almostthm}}\label{mainsect}

In this section we combine the results of Section~\ref{prelimsect} to prove Theorem~\ref{almostthm}. Here, the overarching approach is to use the computation in Theorem~\ref{compthm} for small $n$, and Lemma~\ref{rklemma2} for large $n$. From here onwards, the set $\mcA$ is as defined in \eqref{Adef}, $X$ is as in \eqref{Xdef}, and in the notation of Section~\ref{prelimsect}, we set $k=2$, $k_1=6$, and $k_2=2.16$.

We begin with finding explicit constants for the condition \ref{q0cond}. To do so, we first state a couple of simple lemmas from the literature. Note that in Lemma~\ref{rosserlemma2}, $\pi(x)=\sum_{p\leq x}1$ is the usual prime-counting function.

\begin{lemma}\label{glasbylemma}
    For all $a>1$,
    \begin{equation*}
        \sum_{p\geq a}\frac{1}{p^2}\le \frac{2.22}{a \log a}.
    \end{equation*}
\end{lemma}
\begin{proof}
    For $a\geq 12$, the lemma follows directly from \cite[Lemma 7]{glasby2021most}. The case ${1<a\leq 12}$ can then be verified via a short computation, noting that
    \begin{equation*}
        \sum_{p\geq 2}\frac{1}{p^2}=0.45224\ldots\,.\qedhere
    \end{equation*}
\end{proof}

\begin{lemma}[{\cite[Theorem 1]{rosser1962approximate}}]\label{rosserlemma2}
    For $x>1$, 
    \begin{equation*}
        \pi(x)<\frac{x}{\log x}\left(1+\frac{3}{2\log x}\right).
    \end{equation*}
\end{lemma}

Our explicit version of \eqref{q0cond} is then as follows. In particular, for $N>10^{40}$ we may take $c=0.2$ and $\delta=\frac{1}{4}$. This result could certainly be optimised further, however any improvement would have a negligible impact in our subsequent proof of Theorem~\ref{almostthm}.

\begin{lemma}\label{explicitq0}
    Let $N\geq 10^{40}$, $z=X^{1/6}$ and $y=X^{1/2.16}$. Then,
    \begin{equation*}
        \sum_{\substack{z\le q<y\\q\in \mathcal{P}}}|\mcA_{q^2}|\le 0.2|\mcA|^{3/4}.
    \end{equation*}
\end{lemma}

\begin{proof}
    Since $\mcA$ is a set of consecutive integers,
    \begin{align}\label{Aq2eq1}
        \sum_{\substack{z\le q<y\\q\in \mathcal{P}}}|\mcA_{q^2}|&\le \sum_{\substack{z\le q<y\\q\in \mathcal{P}}}\left(\frac{|\mcA|}{q^2}+1\right)\notag\\
        &=|\mcA|\sum_{\substack{z\le q<y\\q\in \mathcal{P}}}\frac{1}{q^2}+\sum_{\substack{z\le q<y\\q\in \mathcal{P}}}1\notag\\
        &\le \frac{2.22|\mcA|}{z \log z}+\pi(y),
    \end{align}
    where in the last line we have applied Lemma~\ref{glasbylemma}. Now, since $N\geq 10^{40}$, we have
    \begin{equation*}
        z= X^{1/6}\ge N^{1/6}\ge 4.64\cdot 10^6
    \end{equation*}
    but also
    \begin{align*}
        z\geq N^{1/6}=\frac{1}{3^{1/4}}(3N^{2/3})^{1/4}\geq\frac{1}{3^{1/4}}|\mcA|^{1/4}.
    \end{align*}
    Thus, 
    \begin{equation}\label{zbound}
        \frac{2.22|\mcA|}{z\log z}\leq 2.22\cdot3^{1/4}\frac{|\mcA|^{3/4}}{\log(4.64\cdot 10^6)}\leq 0.191|\mcA|^{3/4}.
    \end{equation}
    We now bound $\pi(y)$. Here,
    \begin{equation*}
        y=X^{1/2.16}\ge N^{1/2.16}\geq 3.3\cdot 10^{18}
    \end{equation*}
    and
    \begin{equation*}
        y\leq (N+3N^{2/3})^{1/2.16}\leq 1.01 N^{1/2.16}\leq |\mcA|^{0.69}
    \end{equation*}
    so that Lemma~\ref{rosserlemma2} gives
    \begin{equation}\label{ybound}
        \pi(y)\leq 1.04\frac{y}{\log y}\leq 0.025y\leq 0.025|\mcA|^{0.69}.
    \end{equation}
    Substituting \eqref{zbound} and \eqref{ybound} into \eqref{Aq2eq1} yields
    \begin{equation*}
        \sum_{\substack{z\le q<y\\q\in \mathcal{P}}}|\mcA_{q^2}|\leq 0.191|\mcA|^{3/4}+0.025|\mcA|^{0.69}\leq 0.2|\mcA|^{3/4}
    \end{equation*}
    as required.
\end{proof}

We are now ready to prove Theorem~\ref{almostthm}. To neaten our calculations we only use the computational result in Theorem~\ref{compthm} for $n^3\leq 10^{40}$. In particular, the full strength of Theorem~\ref{compthm} is not required to prove Theorem~\ref{almostthm}.

\begin{proof}[Proof of Theorem~\ref{almostthm}]
    Let $N=n^3\geq 10^{40}$ throughout since smaller values of $N$ are covered by Theorem~\ref{compthm}. We apply Lemma~\ref{explicitq0} and Lemma~\ref{rklemma2} with $k=2$, $z=X^{1/6}$, ${y=X^{1/2.16}}$, and
    \begin{equation*}
        \lambda=k+1-k_2=0.84
    \end{equation*}
    to obtain
    \begin{equation}\label{r2lb1}
        r_2(\mcA)\ge \frac{0.84}{2}S(\mcA, \PP, z)-\frac{1}{2}\sum_{z\le q\le y}\left(1-\frac{\log q}{\log y}\right)S(\mcA_q, \PP, z)-0.084|\mcA|^{3/4}.
    \end{equation}
    We now bound each term in \eqref{r2lb1} separately. To begin with, since $N\geq 10^{40}$, we have $|\mcA|\leq 3N^{2/3}\leq 1.4\cdot 10^{27}$ and thus
    \begin{equation}\label{finalrem}
        0.084|\mcA|^{3/4}\leq  10^{-5}\thinspace\frac{ N^{2/3}}{\log X}.
    \end{equation}
    Next we give a lower bound for the first term in \eqref{r2lb1}. Since $z=X^{1/6}>10^5$, applying Lemma~\ref{linearsieve} with Lemmas~\ref{rosserlem} and~\ref{PrimeProductBound} gives
    \begin{align}\label{sapzlb1}
        S(\mcA, \PP, z)&>|\mcA|V(z)\cdot(f(s)- \varepsilon C_2(\varepsilon)e^2h(s))-\sum_{d<2D}\mu^2(d).\notag\\
        &\geq\frac{e^{-\gamma}(3N^{2/3}-2)}{\log z}\left(1-\frac{1}{2(\log z)^2}\right)(f(s)-110\thinspace\varepsilon e^2h(s))-\sum_{d<2D}\mu^2(d),
    \end{align}
    where $\varepsilon=2.8\cdot 10^{-4}$, $C_2(\varepsilon)=110$ comes from~\cite[Table~1]{BJV24}, and $s=\log D/\log z$ is to be chosen later.
    
    From here, we assume that $s\in[3,4]$ so that by \eqref{hsdef} and Lemma~\ref{fFlem},
    \begin{equation}\label{Csdef}
        f(s)-110\varepsilon e^2 h(s)>C(s):=\frac{2e^{\gamma}\log(s-1)-0.093e^{2-s}}{s}.
    \end{equation}
    To bound the remainder term in \eqref{sapzlb1}, we apply the bound (see \cite[Lemma~4.6]{ramare2013explicit})
    \begin{equation*}
        \sum_{d\leq x}\mu^2(d)\leq \frac{6}{\pi^2}x+0.5\sqrt{x},\qquad x\geq 10,
    \end{equation*}
    which gives
    \begin{equation}\label{sk1eq}
        \sum_{d\leq 2D}\mu^2(d)\leq 1.22D=1.22X^{s/k_1},
    \end{equation}
    noting that $D=z^s\geq z^3\geq 10^{20}$. Substituting \eqref{Csdef} and \eqref{sk1eq} into \eqref{sapzlb1},
    \begin{equation}\label{sapzlb2}
        S(\mcA,\PP,z)>\frac{e^{-\gamma}(3N^{2/3}-2)}{\log z}\left(1-\frac{1}{2(\log z)^2}\right)C(s)-1.22X^{s/k_1}.
    \end{equation}
    Since $k_1=6$, we find that the lower bound in \eqref{sapzlb2} is maximised around $s=3.54$ whereby $C(s)>0.93$ and 
    \begin{equation}\label{finalsapzlb}
        \frac{0.84}{2}S(\mcA,\PP,z)>3.9\frac{N^{2/3}}{\log X}.
    \end{equation}
    Finally we bound the second term in \eqref{r2lb1}. In the notation of Proposition~\ref{newprop35-2} we set $\alpha=0.06$ so that
    \begin{equation}
        D_y=X^{\frac{2}{3}-0.06-\frac{1}{2.16}}>5\cdot 10^5
    \end{equation}
    and it is valid, via Lemma~\ref{PrimeProductBound}, to set $Q=2$ and $\varepsilon=2.8\cdot 10^{-4}$. In particular, Proposition~\ref{newprop35-2} gives
    \begin{equation*}
        \frac{1}{2}\sum_{\substack{q\in \PP\\ z\le q<y}}\left(1-\frac{\log q}{\log y}\right)S(\mcA_q, \PP, z)\le \frac{6 e^{-\gamma}}{2}\left(1+\frac{1}{2(\log D_y)^2}\right)(M_1(X)+M_2(X))+\frac{\mathcal{E}(X)}{2}.
    \end{equation*}
    Here,
    \begin{align*}
        &3e^{-\gamma}\left(1+\frac{1}{(\log D_y)^2}\right)\le 1.7,\\
        &M_1(X)\le \frac{2.04N^{2/3}}{\log X},\\
        &M_2(X)\le (1.9\cdot10^{-8})\frac{N^{2/3}}{\log X}
    \end{align*}
    and
    \begin{equation*}
        \frac{\mathcal{E}(X)}{2}\le 0.15 \frac{N^{2/3}}{\log X}
    \end{equation*}
    computed as per the definitions \eqref{M1def}--\eqref{calEdef} of $M_1(X)$, $M_2(X)$ and $\mathcal{E}(X)$. Combining these bounds,
    \begin{equation}\label{Saqfinalbound}
        \frac{1}{2}\sum_{\substack{q\in \PP\\ z\le q<y}}\left(1-\frac{\log q}{\log y}\right)S(\mcA_q, \PP, z)\leq 3.62\frac{N^{2/3}}{\log X}.
    \end{equation}
    Substituting \eqref{finalrem}, \eqref{finalsapzlb} and \eqref{Saqfinalbound} into \eqref{r2lb1} then yields
    \begin{equation*}
        r_2(\mcA)>0.27\frac{N^{2/3}}{\log X},
    \end{equation*}
    which is positive, as desired.
\end{proof}

\printbibliography

@article{ingham1937difference,
  title={On the difference between consecutive primes},
  author={Ingham, A. E.},
  journal={Q. J. Math.},
  number={1},
  pages={255--266},
  year={1937},
  publisher={Oxford University Press}
}

@article{baker2001difference,
  title={The difference between consecutive primes, II},
  author={Baker, R. C. and Harman, G. and Pintz, J.},
  journal={Proc. Amer. Math. Soc.},
  volume={83},
  number={3},
  pages={532--562},
  year={2001},
  publisher={Cambridge University Press}
}

@article{dudek2016explicit,
  title={An explicit result for primes between cubes},
  author={Dudek, A. W.},
  journal={Funct. Approx. Comment. Math.},
  volume={55},
  number={2},
  pages={177--197},
  year={2016},
  publisher={Adam Mickiewicz University, Faculty of Mathematics and Computer Science}
}

@article{sorenson2025,
  title={An algorithm to verify {L}egendre's conjecture up to $7\cdot 10^{13}$},
  author={J. Sorenson and J. Webster},
  journal={Res. Number Theory},
  volume={11},
  number={4},
  year={2025}
}

@article{BJV24,
  title={An explicit version of {C}hen's theorem and the linear sieve},
  author={M. Bordignon and D. R. Johnston and  V. Starichkova},
  journal={Int. J. Number Theory},
  volume={21},
  number={10},
  pages={2497--2572},
  year={2025},
}

@article{kuhn1954neue,
  title={Neue {A}bsch{\"a}tzungen auf {G}und der {V}iggo {B}runschen {S}iebmethode},
  author={Kuhn, P.},
  journal={Skand. Mat. Kongr.},
  number={12},
  pages={160--168},
  year={1954}
}

@book{greaves2013sieves,
  title={Sieves in number theory},
  author={Greaves, G.},
  volume={43},
  year={2013},
  publisher={Springer-Verlag, Berlin}
}

@article{rosser1962approximate,
  title={Approximate formulas for some functions of prime numbers},
  author={Rosser, J. B. and Schoenfeld, L.},
  journal={Illinois J. Math.},
  volume={6},
  number={1},
  pages={64--94},
  year={1962},
  publisher={Duke University Press}
}

@article{glasby2021most,
  title={Most permutations power to a cycle of small prime length},
  author={Glasby, S. P. and Praeger, C. E. and Unger, W. R.},
  journal={Proc. Edinb. Math. Soc.},
  volume={64},
  number={2},
  pages={234--246},
  year={2021},
  publisher={Cambridge University Press}
}

@article{vanlalngaia2017explicit,
  title={Explicit {M}ertens Sums},
  author={Vanlalngaia, R.},
  journal={Integers},
  volume={17},
  pages={A11},
  year={2017}
}

@article{ramare2013explicit,
  title={From explicit estimates for primes to explicit estimates for the {M}{\"o}bius function},
  author={Ramar{\'e}, O.},
  journal={Acta Arith.},
  volume={157},
  pages={365--379},
  year={2013},
}

@article{mills1947prime,
  title={A prime-representing function},
  author={Mills, W.\ H.},
  journal={Bull. Amer. Math. Soc.},
  volume={53},
  number={6},
  pages={604},
  year={1947}
}

@article{caldwell2005determining,
  title={Determining {M}ills’ {C}onstant and a {N}ote on {H}onaker’s {P}roblem},
  author={Caldwell, C.\ K. and Cheng, Y.},
  journal={J. Integer Seq.},
  volume={8},
  number={05.4.1},
  year={2005}
}

@article{cully2022explicit,
  title={Explicit interval estimates for prime numbers},
  author={Cully-Hugill, M. and Lee, E.},
  journal={Math. Comp.},
  volume={91},
  number={336},
  pages={1955--1970},
  year={2022}
}

@misc{cully2023explicit, 
    author = {M. Cully-Hugill}, 
    title = {\textit{Explicit estimates for the distribution of primes}. {P}h{D} thesis, {{U}niversity of {N}ew {S}outh {W}ales {C}anberra}}, 
    year = {2023}, 
}

@article{iwaniec1979primes,
  title={Primes in short intervals},
  author={Iwaniec, H. and Jutila, M.},
  journal={Ark. Mat.},
  volume={17},
  number={1},
  pages={167--176},
  year={1979},
}

@article{baker1996difference,
  title={The difference between consecutive primes},
  author={Baker, R. C. and Harman, G.},
  journal={Proc. London Math. Soc.},
  volume={3},
  number={2},
  pages={261--280},
  year={1996},
}

@article{dudek2026almost,
  title={Almost primes between all squares},
  author={Dudek, A. W. and Johnston, D. R.},
  journal={J. Number Theory},
  volume={278},
  pages={726--745},
  year={2026},
}

@article{richert1969selberg,
  title={Selberg's sieve with weights},
  author={Richert, H.-E.},
  journal={Mathematika},
  volume={16},
  number={1},
  pages={1--22},
  year={1969},
}

@book{harman2007prime,
  title={Prime-{D}etecting {S}ieves},
  author={Harman, G.},
  year={2007},
  publisher={Princeton University Press, New Jersey}
}

@article {brillhart1975,
    AUTHOR = {Brillhart, J. and Lehmer, D. H. and Selfridge, J. L.},
     TITLE = {New primality criteria and factorizations of {$2\sp{m}\pm 1$}},
   JOURNAL = {Math. Comp.},
  FJOURNAL = {Mathematics of Computation},
    VOLUME = {29},
      YEAR = {1975},
     PAGES = {620--647},
}

@article{cramer1936order,
  title={On the order of magnitude of the difference between consecutive prime numbers},
  author={Cram{\'e}r, H.},
  journal={Acta Arith.},
  volume={2},
  pages={23--46},
  year={1936},
}

@article{elsholtz2020unconditional,
  title={Unconditional prime-representing functions, following {M}ills},
  author={Elsholtz, C.},
  journal={Amer. Math. Monthly},
  volume={127},
  number={7},
  pages={639--642},
  year={2020}
}

\end{document}